\documentclass{amsart}

\title{On Levi flat hypersurfaces with transversely affine foliation}
\date{\today}
\date{\today}

\author{Masanori Adachi}
\address{Department of Mathematics, Faculty of Science, Shizuoka University.  836 Ohya, Suruga-ku, Shizuoka 422-8529, Japan.}
\email{adachi.masanori@shizuoka.ac.jp}

\author{S\'everine Biard}
\address{Universit\'e Polytechnique Hauts-de-France, Ceramaths, FR CNRS 2956, F-59313 Valenciennes, France}

\email{severine.biard@uphf.fr}

\thanks{We acknowledge the support from Watanabe Trust Fund of the University of Iceland. 
The first author is partially supported by a JSPS KAKENHI Grant Number JP18K13422.}

\usepackage{amsmath,amsthm,amssymb,latexsym}
\usepackage[abbrev]{amsrefs}
\usepackage{color}
\usepackage[normalem]{ulem}
\usepackage{hyperref}
\usepackage{url}
\newtheorem{Theorem}{Theorem}
\newtheorem{Proposition}[Theorem]{Proposition}
\newtheorem{Definition}[Theorem]{Definition}
\newtheorem{Lemma}[Theorem]{Lemma}

\theoremstyle{remark}
\newtheorem{Remark}[Theorem]{Remark}
\newtheorem{Example}[Theorem]{Example}

\newcommand\C{\mathbb{C}}  
\newcommand\R{\mathbb{R}}

\newcommand\Z{\mathbb{Z}}
\newcommand\D{\mathbb{D}}
\newcommand\PP{\mathbb{P}}  
\renewcommand\Re{\operatorname{Re}}

\newcommand\Res{\operatorname{Res}}
\newcommand\tr{\operatorname{tr}}
 
\newcommand{\pa}{\partial}
\newcommand{\opa}{\overline\pa}
\newcommand{\ol}{\overline }

\begin{document}

\begin{abstract}We prove the non-existence of real analytic Levi flat hypersurface whose complement is 1-convex and Levi foliation is transversely affine in a compact K\"ahler surface. 
\end{abstract}

\maketitle

\indent Among the most comprehensible structures in the study of geometric foliations, the class of foliations with transverse affine structure is one of the first studied. Those foliations are called \textit{transversely affine foliations} and are characterized by the change of coordinates of their foliated atlas to be affine in the transverse direction. They have been studied first by  F\'edida and Furness \cite{Fedida-Furness} and then by Seke \cite{Seke} for smooth foliations of real codimension one. The case of codimension one transversely affine holomorphic foliations with singularities have been introduced by  Sc\'ardua \cite{Scardua}. In \cite{Cousin-Pereira},  Cousin and Pereira described the transverse affine structure in terms of a meromorphic connection of the normal bundle to the foliation and give a classification of those foliations in complex projective manifolds.

In this paper, we study transversely affine holomorphic foliations coming from a real analytic \textit{Levi flat} hypersurface in complex K\"ahler surfaces. 
A smooth real hypersurface $M$ in a complex manifold of dimension $\geq 2$ is \textit{Levi flat} if its Levi form vanishes identically or equivalently, if it admits a smooth foliation by complex hypersurfaces, called the \textit{Levi foliation} of $M$. 
When $M$ is real analytic, the Levi foliation extends to a holomorphic foliation on a neighborhood of $M$ (cf. \cite{Cartan}), and therefore may be characterized as an invariant real hypersurface of a codimension one holomorphic foliation. 
The existence of such an invariant closed real hypersurface is expected to pose a strong restriction on the foliations and ambient spaces, and their classification are of great interest, in particular, when the foliation is transversely affine.

A fundamental tool in the study of the (non-)existence of Levi flat hypersurfaces the last 20 years is the pseudoconvexity of their complements,
which are not only locally pseudoconvex by definition but also could show stronger pseudoconvexity in certain circumstances.
Motivated by the conjecture of an exceptional minimal set for holomorphic foliations of the complex projective plane $\mathbb{CP}^2$ (cf. \cite{Cerveau}), 
Lins Neto \cite{LinsNeto} proved the non-existence of real analytic closed Levi flat hypersurface in $\mathbb{CP}^n$, $n\geq 3$, based on the Steinness of locally pseudoconvex proper subdomains of $\mathbb{CP}^n$ (cf. \cite{Takeuchi}). 
This result has then been generalized by Ohsawa \cite{Ohsawa2007} in compact K\"ahler manifolds in dimension $\geq 3$ in which the non-existence was shown when the complement of the real analytic Levi flat hypersurface is assumed to be Stein for instance. 
Further constraints on the complex manifold $X$, in particular when $X$ is a surface,  are less known, for instance, the  (non-)existence of real analytic closed Levi flat hypersurface in $\mathbb{CP}^2$ is still an open problem (cf. \cite{Iordan-Matthey}). We refer the interested reader to \cite{OhsawaBook} for background and recent studies.

In the quest of non-existence theorems, the geometry of the complement of a compact Levi flat hypersurface sparks the interest. Brunella \cite{Brunella} turned the assumption on the complement to an assumption on the normal bundle $N_\mathcal{F}$ to the Levi foliation and proved the non-existence of real analytic closed Levi flat hypersurface whose normal bundle to the Levi foliation has positive curvature in compact K\"ahler manifolds of dimension $\geq 3$. 
Recently, Canales \cite{Canales} proved that, to have a normal bundle $N_\mathcal{F}$ of positive curvature, it is enough to assume the Levi foliation $\mathcal{F}$ to have chaotic dynamics or equivalently to not admit a transverse invariant measure in (abstract) compact Levi flat 3-manifolds \cite{Canales}*{Theorem 3.4}. 
Thanks to it, Canales deduced a dynamical property of Levi flat hypersurfaces:  
in complex algebraic surfaces, a real analytic closed Levi flat hypersurface admitting a transverse affine structure has a transverse invariant measure \cite{Canales}*{Theorem 6.8}. Combined with a theorem of Ghys \cite{Ghys91}*{Th\'eor\`eme 2.8}, whose proof works for (real) transversely affine foliations, such a Levi flat hypersurface contains a compact leaf or is defined by a closed 1-form. Her result echoes Cousin and Pereira's works \cite{Cousin-Pereira} on the classification of algebraic transversely affine foliations of codimension one.

\medskip
We are interested here in real analytic closed Levi flat hypersurfaces whose complement in a compact complex manifold is \emph{1-convex}, namely, admitting a plurisubharmonic exhaustion function which is strictly plurisubharmonic away from a compact subset. 
From a result of Ohsawa \cite{Ohsawa2007}*{Theorem 0.1}, such a Levi flat hypersurface may exist only in 
surfaces or non-K\"ahler manifolds of dimension $\geq 3$, and only a few examples are known:
Holomorphic disk bundles in certain ruled surfaces (cf. \cite{Ohsawa1982}, \cite{Diederich-Ohsawa}); Hyperbolic torus bundles in hyperbolic Inoue surfaces (cf. \cite{Zaffran}); Reeb foliations in Hopf manifolds of arbitrary dimension (cf. \cite{Nemirovski}).
The first examples live in K\"ahler surfaces but their Levi foliations are not transversely affine; 
In the second examples, their Levi foliations are transversely affine, but they are in non-K\"ahler manifolds (see \S\ref{sect:example} for details). 
The main result of this paper shows that the K\"ahler structure of the ambient surface and the transversely affine structure of the Levi foliation are mutually exclusive when the complement is 1-convex.

\begin{Theorem}
\label{thm:main}
Let $X$ be a compact K\"ahler surface and $M$ a real analytic closed Levi flat hypersurface whose complement is 1-convex. Then, the Levi foliation of $M$ cannot be transversely affine. 
\end{Theorem}

The main part of this theorem can be seen as a generalization of Canales' result \cite{Canales}*{Theorem 6.8} mentioned above.
Our contribution is to relax the assumption on the ambient manifold $X$ from algebraic to K\"ahler.
In \cite{Canales}, the projectivity of $X$ was used to extend the Levi foliation of $M$ to a  holomorphic foliation $\widetilde{\mathcal{F}}$ on entire $X$ \cite[Proposition 5.13]{Canales}, and to analyze transversely affine structure of the Levi foliation as degenerate transversely affine structure using a global meromorphic section of $N_{\widetilde{\mathcal{F}}}$  \cite[\S6.2.2 and \S6.3]{Canales}. 
To bypass these technical points, we apply the theory of integrable connection of vector bundles over complex manifolds (cf. \cite{Deligne}). 

Also, compared to \cite{Canales}*{Theorem 6.8}, we achieve a non-existence result instead of deducing a dynamical property. 
The main idea is to contradict the 1-convexity of the complement by showing the existence of Levi flat hypersurfaces close enough to the original one. In order to do that, we are going to build a closed 1-form defining the Levi foliation on $M$. 

\medskip
The first two sections are focused on preliminaries such as the notion of integrable connection in \S\ref{sect:connection}, and transversely affine foliations in \S\ref{sect:trans_affine}.
Theorem \ref{thm:main} is proved in \S\ref{sect:proof}. 
Some examples are provided in \S\ref{sect:example}.

\section{Integrable connections}
\label{sect:connection}

In this section, we briefly recall notions around integrable connections of holomorphic vector bundles over complex manifolds. 
We restrict ourselves to the case of line bundles, and refer the reader to \cites{Deligne,Malgrange,Ohtsuki} for backgrounds and details. 

Let $X$ be a complex manifold and $L$ a holomorphic line bundle over $X$.
We denote by $\mathcal{O}_X(L)$ the sheaf of germs of holomorphic sections of $L$. 
A \textit{holomorphic connection} $\nabla$ on $L$ is a $\C$-linear sheaf morphism
$$\nabla \colon \mathcal{O}_X(L)\to  \Omega^1_X \otimes_{\mathcal{O}_X} \mathcal{O}_X(L)$$
that satisfies the Leibniz rule $\nabla(fs)=df \otimes s+f \nabla s$ for any holomorphic function $f$ and holomorphic section $s$ of $L$ defined on a common open set. Here $\Omega^1_X$ denotes the sheaf of germs of holomorphic 1-form.

When we trivialize $L$ locally by a non-vanishing section $e_\alpha$ over an open set $U_\alpha$, the connection is expressed by a holomorphic 1-form $\gamma_\alpha$ on $U_\alpha$ such that if $s$ is a holomorphic section of $L$ over $U_\alpha$, then in the holomorphic trivialization
$\nabla s = (ds_\alpha + \gamma_\alpha s_\alpha) e_\alpha $ where $s = s_\alpha e_\alpha$.
Note that the connection 1-form $\gamma_\alpha$ transforms as $\gamma_{\beta}=\gamma_{\alpha}+ g_{\alpha\beta}^{-1} dg_{\alpha\beta}$
when we take another local trivialization $e_\beta$ over $U_\beta$ where $e_\beta = g_{\alpha\beta} e_\alpha$ for $g_{\alpha\beta} \in \mathcal{O}^*(U_\alpha \cap U_\beta)$. 
Conversely, a holomorphic connection on $L$ is given by the cocycle $(g_{\alpha\beta})$ representing $L$ with respect to an open covering $\{ U_\alpha \}$ and 
the collection of connection 1-forms $(\gamma_\alpha)$ enjoying the transformation rule.

\begin{Remark}\label{rem:connection_difference}
When two holomorphic connections $\nabla_1$ and $\nabla_2$ on $L$ are given, their difference $\nabla_1 - \nabla_2$ gives a well-defined holomorphic 1-form on $X$. Denote the connection forms of $\nabla_1$ and $\nabla_2$ by $(\gamma_{1, \alpha})$ and $(\gamma_{2, \alpha})$ respectively. Since they obey the same transformation rule,
$(\gamma_{1, \alpha} - \gamma_{2,\alpha})$ glue together and define a holomorphic 1-form on $X$.
\end{Remark}

The curvature $\Theta$ of $\nabla$ is a holomorphic 2-form such that for every locally defined holomorphic section $s$ of $L$, $\nabla^2s=\Theta \otimes s$, where $\nabla$ is extended on $L$-valued holomorphic forms by the Leibniz rule. 
Locally, $\Theta= d\gamma_\alpha$.
A \textit{flat} or \textit{integrable} connection $\nabla$ is one whose curvature vanishes identically, i.e.; $\Theta=0$.\\\medskip

Let $Z = \sum_{j=1}^N Z_j$ be a normal crossing divisor in $X$. By choosing a sufficiently fine coordinate covering $\lbrace{U_\beta}\rbrace$ of $X$, we may assume that each component $Z_j$ has a holomorphic defining function $f_{\beta j}$ on $U_\beta$.
We denote the support of a divisor $Z$ by $|Z|$. 
In \S\ref{sect:proof}, we shall use holomorphic connection on $L$ over $X \setminus |Z|$ having nice singularity along $Z$ in the following sense.

\begin{Definition}
A logarithmic connection $\nabla$ on $L$ along $Z$ is a $\C$-linear sheaf morphism
$$\nabla\colon \mathcal{O}_X(L)\to \Omega^1_X(\log Z)\otimes_{\mathcal{O}_X}  \mathcal{O}_X(L),$$
such that
$$\nabla(fs)=df\otimes s + f \nabla s,$$
for any local section $f\in\mathcal{O}_X$, $s\in\mathcal{O}_X(L)$ defined on a common open set, where $\Omega^1_X(\log Z)$ denotes the sheaf of germs of meromorphic 1-forms on $X$ with logarithmic poles along $Z$.\\\medskip
\end{Definition}

When we trivialize $L$ locally by a non-vanishing section $e_\alpha$ over an open set $U_\alpha$,  the logarithmic connection $\nabla$ is then expressed by a meromorphic 1-form $\widetilde{\gamma_\alpha}$ on $U_\alpha$ such that
$\nabla s = (ds_\alpha + \widetilde{\gamma_\alpha} s_\alpha) e_\alpha $ where $s = s_\alpha e_\alpha \in \mathcal{O}_X(L)$.
Along each component $Z_j$, $\widetilde{\gamma_\alpha}$ can be expressed as
\begin{align}\label{logconnection}
\widetilde{\gamma_\alpha}=a_{\alpha j} \dfrac{d f_{\alpha j}}{f_{\alpha j}} + b_{\alpha j},
\end{align}
where $a_{\alpha j}\in \mathcal{O}(U_\alpha)$, $b_{\alpha j}$ is a meromorphic $1$-form on $U_\alpha$ with
logarithmic pole along $\bigcup_{i\neq j}Z_i$ and $f_{\alpha j}$ a defining function of $Z_j$ on $U_\alpha$.
\begin{Definition}
Under the notations above, we define
$$\Res_{Z_j}(\nabla):= a_{\alpha j}\vert_{Z_j}.$$
This is a holomorphic function on $Z_j\cap U_\alpha$ that does not depend on the choice of $f_{\alpha j}$ nor on the representation \eqref{logconnection} of $\widetilde{\gamma_\alpha}$.
Hence, $\Res_{Z_j}(\nabla)$ defines a holomorphic, hence, constant function on $Z_j$, 
whose value we call the \emph{residue of $\nabla$ along $Z_j$}.
\end{Definition}

\medskip
Thanks to the following theorem, our consideration in \S\ref{sect:proof} will be reduced to the case of integrable logarithmic connections. 

\begin{Theorem}[{Deligne \cite{Deligne}*{Proposition II.5.4}} attributed to Manin, cf. {\cite{Malgrange}*{Theorem IV.4.4}}]
Let $X$ be a complex manifold and $Z$ a normal crossing divisor in $X$. 
Consider a holomorphic line bundle $L$ over $X \setminus |Z|$ with an integrable connection $\nabla$.
Then, there exists a holomorphic line bundle $\widetilde{L}$ over $X$ with an integrable logarithmic connection $\widetilde{\nabla}$ along $Z$ extending $\nabla$ on $L$.
\label{thm:deligne}
\end{Theorem}

The following theorem gives a relationship between the residues of integrable logarithmic connection $\nabla$ of $L$ and the first Chern class of $L$. 
We will use this theorem in \S\ref{sect:proof} combining with the expression of the first Chern class of $L$ in terms of hermitian metrics of $L$.

\begin{Theorem}[Ohtsuki \cite{Ohtsuki}]
\label{thm:ohtsuki}
Let $L$ be a holomorphic line bundle over a compact complex manifold $X$ and $\nabla$ an integrable logarithmic connection of $L$ along a simple normal crossing divisor $Z = \sum_{j=1}^N Z_j$. 
Then, it holds that
$$c_1(L)=-\sum_{j=1}^N(\Res_{Z_j} \nabla)[Z_j] \quad \text{in $H^{1,1}(X)$}$$
where $c_1(L)$ is the first Chern class of $L$ and $[Z_j]$ is the current of integration along smooth hypersurface $Z_j$.
\end{Theorem}
Note that, in our notation, $c_1(L)$ and $[Z_j]$ stand for the Dolbeault cohomology classes of degree $(1,1)$. We will later regard them as de Rham cohomology classes, and use notation $c_1(L)_k$ and $[Z_j]_k$ when we work on $H^2(X; k)$, $k = \R$ or $\C$. 

\section{Transversely affine foliation}
\label{sect:trans_affine}

Transversely affine foliations are a particular case of transversely homogeneous foliations of codimension one. Such foliations are characterized by the existence of a transversely affine \textit{structure} that allows to describe the defining submersion of the foliation.

In this paper, a foliation can be singular, and we say \textit{smooth} or \textit{non-singular} foliations when we consider those without singularities. For transversely affine structure, we will always refer to \textit{non-degenerate} or \textit{regular} one (see Remark \ref{Remark} at the end of the section). 
So, we will omit this adjective unless otherwise stated.\\\medskip

We start this section by recalling the definition and a characterization of a real transversely affine foliation in a real manifold. We refer to \cite[Chap.III]{Godbillon} for details. 
 
\begin{Definition}
Let $M$ be a smooth manifold, and $\mathcal{F}$ a real codimension one smooth foliation of $M$. 
The foliation $\mathcal{F}$ of $M$ is said to be \emph{transversely affine} 
if $\mathcal{F}$ is given by a foliated atlas $\{(U_\alpha, (x_\alpha, t_\alpha))\}$ 
such that the change of transverse coordinates on non-empty $U_\alpha \cap U_\beta$ is given by an affine map
\[
t_\beta = a_{\alpha\beta} t_\alpha + b_{\alpha\beta}, \quad a_{\alpha\beta} \in \R^*, b_{\alpha\beta} \in \R.
\]
\end{Definition}

A transversely affine foliation is also characterized by a transversely affine structure that is given by differential forms.

\begin{Definition}
Let $M$ be a smooth manifold, and $\mathcal{F}$ a real codimension one smooth foliation of $M$. Assume that $\mathcal{F}$ is transversely oriented. 
A \emph{transversely affine structure} for $\mathcal{F}$ is given by a couple of 1-forms $(\omega,\eta)$ on $M$ such that
\begin{itemize}
\item $\omega$ is non-vanishing and defines $\mathcal{F}$;
\item $d\omega=\eta\wedge \omega$;
\item The form $\eta$ is $d$-closed.
\end{itemize}
Moreover, we say that two couples $(\omega,\eta)$ and $(\omega',\eta')$ are equivalent if there exists a non-vanishing smooth function $f$  on $M$ such that $\omega'=f\omega$ and $\eta'=\eta+\frac{df}{f}$.\\\medskip
\end{Definition}

\begin{Example}[Hyperbolic torus bundle, cf. \cite{Ghys-Sergiescu}*{\S1}]
\label{ex:hyp_torus}
A classical example of real transversely affine foliations is the model foliations on hyperbolic torus bundles. Let $\mathbb{T}^2=\R^2/\mathbb{Z}^2$ and $A\in SL(2,\mathbb{Z})$. Denote by $\mathbb{T}^3_A=\R \times \mathbb{T}^2/\thicksim$ a torus bundle over the circle where 
$(t,p)\thicksim(t+1,Ap)$ for all $(t,p)\in \R \times \mathbb{T}^2 $. We say that $\mathbb{T}^3_A$ is \emph{hyperbolic} if $\vert \tr A \vert >2$.

Assume, for simplicity, that $\tr A>2$. The automorphism $A$ on $\mathbb{T}^2$ has 2 eigenvectors with irrational slopes. 
We obtain a linear foliation on $\mathbb{T}^2$ from the foliation of $\mathbb{R}^2$ by parallel lines directed by one of these eigenvectors and by passing to the quotient.
Taking product with $\mathbb{R}$, a foliation by parallel planes on $\R \times \mathbb{T}^2$ is induced. This foliation on $\R \times \mathbb{T}^2$ is invariant by the equivalence $\thicksim$, and gives a foliation by cylinders on $\mathbb{T}^3_A$, that is called a \emph{model foliation}. Hence, there are 2 foliations on $\mathbb{T}^3_A$, one for each eigenvectors.

These foliations are transversely affine (\cite{Ghys-Sergiescu}*{Proposition 3}):
We consider the foliation associated to an eigenvector $e_1$ of the matrix $A$ with the eigenvalue $\lambda > 1$. 
Denote by $e_2$ another eigenvector with the eigenvalue $\lambda^{-1}$. 
Let $(t, (x,y))$ be the coordinates on $\R \times \R^2$ with $\mathbb{R}^2=\lbrace{xe_1+ye_2\mid x,y\in\mathbb{R}}\rbrace$.
The transversely affine structure $(\omega, \eta)$ of the foliation is then given by $\omega=\lambda^tdy$ and $\eta= \log \lambda dt$. Both forms pass to quotient on $\mathbb{T}^3_A$ and verify $d\omega=\eta\wedge \omega$ and $d\eta=0$.\\\medskip
\end{Example}

Below, we recall an equivalent definition for transversely affine holomorphic foliations in a complex manifold.

\begin{Definition}\label{def transAffH}
Let $X$ be a complex manifold and $\mathcal{F}$ a codimension one non-singular holomorphic foliation on $X$. We say that $\mathcal{F}$ is \emph{transversely affine} if there exists a 
foliated atlas $\{(U_\alpha, (z_\alpha, w_\alpha))\}$ of $\mathcal{F}$ such that on non-empty $U_\alpha\cap U_\beta$ the change of transverse coordinates is given by
\[
w_\beta = a_{\alpha\beta} w_\alpha + b_{\alpha\beta}, \quad a_{\alpha\beta} \in \C^*, b_{\alpha\beta} \in \C.
\]
\end{Definition}
A \textit{singular} holomorphic foliation $\mathcal{F}$ on a complex manifold $X$ is said to be \emph{transversely affine} if it is transversely affine on $X\setminus Sing(\mathcal{F})$ where $Sing(\mathcal{F})$ is the set of singularities of $\mathcal{F}$.\\\medskip

We also have a characterization in terms of transversely affine structure as for real foliations. 

\begin{Proposition}[Sc\'ardua {\cite[Proposition 6.1]{Scardua}}]
\label{prop:scardua}
Let $\mathcal{F}$ be a holomorphic singular codimension one foliation on a complex manifold $X$. The foliation $\mathcal{F}$ is transversely affine on $X$ if and only if there exist an open cover $\{ U_\alpha \}$ and a collection of holomorphic 1-forms $(\omega_\alpha,\eta_\alpha)$ on $U_\alpha$ where 
\begin{itemize}
\item $\omega_\alpha$ defines $\mathcal{F}$ on $U_\alpha$;
\item $d\omega_\alpha =\eta_\alpha \wedge \omega_\alpha$;
\item The form $\eta_\alpha$ is $d$-closed;
\item If $U_\alpha\cap U_\beta\neq \emptyset$, then $\eta_{\alpha}=\eta_\beta+ g_{\alpha\beta}^{-1} dg_{\alpha\beta}$
where $g_{\alpha\beta}\in\mathcal{O}^*(U_\alpha\cap U_\beta)$ satisfies $\omega_{\alpha}=g_{\alpha\beta}\omega_{\beta}$ on $U_\alpha\cap U_\beta$.
\end{itemize}
Moreover, two collections $(\omega_\alpha',\eta_\alpha')$ and $(\omega_\alpha,\eta_\alpha)$ define the same affine transverse structure for $\mathcal{F}$ if and only if 
there exists a collection of $h_\alpha\in\mathcal{O}^*(U_\alpha)$ such that $\omega_\alpha'=h_\alpha\omega_\alpha$ and $\eta_\alpha'=\eta_\alpha+ h_\alpha^{-1} dh_\alpha$.\\\medskip
\end{Proposition}

We will use the fact that a transversely affine structure of $\mathcal{F}$ induces an integrable connection on the normal bundle $N_\mathcal{F}$ in \S\ref{sect:proof}.
Let us recall first the relationship between defining 1-forms of the foliation and its normal bundle.
Let $\mathcal{F}$ be a codimension one holomorphic foliation on a complex manifold $X$ with normal bundle $N_\mathcal{F}$. 
Such a foliation is given locally by an open covering $\lbrace U_\alpha\rbrace$ of $X$ 
and a collection of holomorphic 1-forms $\omega_\alpha$ defining $\mathcal{F}$ on $U_\alpha$ and patched by $\omega_{\alpha}=g_{\alpha\beta}\omega_{\beta}$ on $U_\alpha\cap U_\beta$, 
where $g_{\alpha\beta}\in\mathcal{O}^*(U_\alpha\cap U_\beta)$. 
The normal bundle $N_{\mathcal{F}}$ is represented by the cocycle $(g_{\alpha\beta})$. 
We remark that the normal bundle $N_{\mathcal F}$ is well-defined as a holomorphic line bundle
for \emph{singular} holomorphic foliations of codimension one (cf. \cite{OhsawaBook}*{\S5.1.1}).

\begin{Lemma}
\label{lem:transAff to connection}
Let $\mathcal{F}$ be a codimension one holomorphic foliation on a complex manifold $X$ with normal bundle $N_\mathcal{F}$. 
Assume that $\mathcal{F}$ is transversely affine. Then, there exists an integrable connection $\nabla$ on $N_\mathcal{F}$. 
\end{Lemma}

\begin{proof}
From Proposition \ref{prop:scardua}, we have an open cover $\{U_\alpha\}$ and 
the collection of $(\omega_\alpha,\eta_\alpha)$ of holomorphic 1-forms on $U_\alpha$. 
As explained above, $N_\mathcal{F}$ is represented by the cocycle $(g_{\alpha\beta})$ 
where $\omega_{\alpha}=g_{\alpha\beta}\omega_{\beta}$, $g_{\alpha\beta}\in\mathcal{O}^*(U_\alpha\cap U_\beta)$. 
From the transformation rule $\eta_{\alpha}=\eta_\beta+ g_{\alpha\beta}^{-1} dg_{\alpha\beta}$, we see that the collection of $\gamma_\alpha := -\eta_\alpha$ defines a holomorphic connection $\nabla$ on $N_\mathcal{F}$.
Since each $\eta_\alpha$ is a closed 1-form, the curvature of $\nabla$ is $\Theta=d\gamma_\alpha=0$ on any $U_\alpha$.  
Hence, $N_\mathcal{F}$ has an integrable connection $\nabla$ over $X$.
\end{proof}

\section{Proof of Theorem \ref{thm:main}}
\label{sect:proof}
\begin{proof}[Proof of Theorem \ref{thm:main}]
The proof is by contradiction. Suppose that the Levi foliation of $M$ is transversely affine as a real analytic foliation.
We may assume that $M$ is connected by choosing its connected component, 
and also that $M$ is oriented by taking a suitable double cover of $X$.

First, we shall use the 1-convexity of the complement of $M$. 
Applying theorems of Grauert and of Narasimhan (see, for instance, \cite{Peternell}) to the Remmert reduction of $X \setminus M$, we obtain a real analytic plurisubharmonic exhaustion $\varphi$ of $X \setminus M$ which is strictly plurisubharmonic away from the maximal compact analytic set $A \subset X \setminus M$. 
We may perturb $\varphi$ in Whitney $C^2$-topology to be a smooth proper strictly plurisubharmonic Morse function $\widetilde{\varphi} \colon X \setminus (M \cup A) \to (m, +\infty)$ where $m := \inf \widetilde{\varphi}$.
Using this function $\widetilde{\varphi}$, we can extend the Levi foliation of $M$ to a possibly singular holomorphic foliation $\mathcal{F}$ in $X \setminus A$ (see a Bochner--Hartogs type extension theorem in \cite{Canales}*{\S5} or \cite{Ivashkovich}*{\S4.3}).
By Hironaka's desingularization theorem, we have another compact K\"ahler surface $Y$, a proper holomorphic map $\pi \colon Y \to X$ and a simple normal crossing divisor $Z = \sum_{j=1}^N Z_j$ of $Y$ such that $\pi$ is a biholomorphism from $Y \setminus |Z|$ to $X \setminus A$.
We identify $M \simeq \pi^{-1}(M) \subset Y$ and regard $\mathcal{F}$ as a holomorphic foliation on $Y \setminus |Z|$ via $\pi$.

From the real analyticity of $M$, the extension of Levi foliation $\mathcal{F}$ is transversely affine on some small neighborhood $W$ of $M$ (see \cite{Canales}*{Proposition 6.3}). 
More precisely, we may find a foliated atlas $\{ (U_\alpha, (z_\alpha, w_\alpha = t_\alpha + iu_\alpha))\}$ of $\mathcal{F}_{|W}$ such that
\begin{itemize}
\item The foliation $\mathcal{F}$ is given by $dw_\alpha$ on $U_\alpha$;
\item The foliated chart $(z_\alpha, w_\alpha)$ maps $U_\alpha$ to $\D^2$ biholomorphically, and $M \cap U_\alpha$ to $\D \times I$; 
\item On $U_\alpha \cap U_\beta$, the change of transverse coordinates is of the form
\begin{equation}
\label{eq:affine}
w_\beta = a_{\alpha\beta} w_\alpha + b_{\alpha\beta}, \quad a_{\alpha\beta} > 0, b_{\alpha\beta} \in \R. \tag{A}
\end{equation}
Namely,
\[
t_\beta = a_{\alpha\beta} t_\alpha + b_{\alpha\beta}, \quad
u_\beta = a_{\alpha\beta} u_\alpha.
\]
\end{itemize}
Here we wrote $\D := \{ z \in \C \mid |z| < 1\}$ and $I := \{ x \in \R \mid |x| < 1\}$.

Next, we shall construct an integrable logarithmic connection on an extension of the normal bundle $N_{\mathcal{F}}$ over $Y$.
By Lemma \ref{lem:transAff to connection},
the normal bundle $N_{\mathcal{F}_{|_W}}$ admits an integrable connection.
Notice that integrable connections of a fixed vector bundle verify the following two properties:
\begin{enumerate}
\item (Uniqueness property) For any integrable connections $\nabla_1$, $\nabla_2$  of $N_{\mathcal{F}}$ over a connected open subset $U \subset Y \setminus |Z|$, 
$\nabla_1 = \nabla_2$ over $U$ if $\nabla_1 = \nabla_2$ over some non-empty open subset $V \subset U$.  
\item (1-Hartogs-type extension lemma) Let $i \colon \D^2 \to Y\setminus |Z|$ be a holomorphic embedding of the bidisk and denote by $H_\varepsilon := \{ (z,w) \in \D^2 \mid \text{$|z| < \varepsilon$ or $|w| > 1 - \varepsilon$} \}$ the Hartogs figure. Then, any integrable connection $\nabla$ of $N_{\mathcal{F}}$ over $i(H_\varepsilon)$ extends to $i(\D^2)$. 
\end{enumerate}
The first point is reduced to the identity theorem for holomorphic 1-forms since the difference $\nabla_1 - \nabla_2$ is a closed holomorphic 1-form over $U$ (see Remark \ref{rem:connection_difference}).
The second point is also reduced to the Hartogs extension theorem for holomorphic 1-forms since $N_{\mathcal{F}}$ is holomorphically trivial over $i(\D^2)$ and integrable connections are identified with closed holomorphic 1-forms. 
Hence, we can apply the Bochner--Hartogs type extension theorem (cf. \cite{Ivashkovich}*{\S4.3}) to the integrable connection of $N_{\mathcal{F}}$ over $W$ and extend it to $Y \setminus |Z|$.
Theorem \ref{thm:deligne} yields a holomorphic line bundle $L$ over $Y$ which extends $N_\mathcal{F}$ and an integrable logarithmic connection $\nabla$ of $L$ along $Z$. 

We shall construct a smooth hermitian metric of $N_\mathcal{F}$ which is flat around $M$
by exploiting the identity in Theorem \ref{thm:ohtsuki}.
Since $Y$ is K\"ahler, the identity holds in $H^2(Y; \C)$ via the Hodge decomposition $H^2(Y;\C) \simeq H^{2,0}(Y) \oplus  H^{1,1}(Y) \oplus  H^{0,2}(Y)$:
\[
c_1(L)_\C = - \sum_{j=1}^N (\Res_{Z_j} \nabla) [Z_j]_\C \quad \text{in $H^{2}(Y;\C)$.}
\]
Since both $c_1(L)_\C$ and $[Z_j]_\C$ can be represented by real classes, we have
\[
c_1(L)_\R = - \sum_{j=1}^N \Re(\Res_{Z_j} \nabla) [Z_j]_\R \quad \text{in $H^{2}(Y;\R)$.}
\]

Fix an arbitrary smooth hermitian metric $h_0$ of $L$. 
Also, for each $j$, we take smooth hermitian metric $h_j$ of the line bundle associated with $Z_j$ so that the support of the curvature form $\Theta_{h_j}$ does not intersect with the neighborhood $W$ of $M$. 
Since $Y$ is K\"ahler, we may use the $\pa\opa$-lemma to obtain a smooth real function $\psi$ on $Y$  such that 
\[
\frac{i}{2\pi}\Theta_{h_0} = - \sum_{j=1}^N  \Re(\Res_{Z_j} \nabla) \frac{i}{2\pi}\Theta_{h_j} + i\pa\opa \psi.
\]
Letting $h := h_0 e^{2\pi\psi}$, we obtain a smooth hermitian metric $h$ of $L$ 
whose curvature form
\[
\Theta_h =  - \sum_{j=1}^N  \Re(\Res_{Z_j} \nabla) \Theta_{h_j} 
\]
vanishes away from the supports of $\Theta_{h_j}$. 
Since $L$ is isomorphic to $N_{\mathcal{F}}$ over $Y \setminus |Z|$, 
the restriction of $h$ yields a flat hermitian metric of $N_{\mathcal{F}}$ over $W$.

Using this flat metric $h$, we shall show that the Levi foliation of $M$ is defined by a closed 1-form on $M$.  
We write $h(\frac{\pa}{\pa w_\alpha}, \frac{\pa}{\pa w_\alpha}) = h_\alpha$ in each foliated chart $U_\alpha$ of $\mathcal{F}_{|_W}$.
Consider a smooth 1-form $\omega := \sqrt{h_\alpha(z_\alpha, t_\alpha)} dt_{\alpha}$ on $M$ defining the Levi foliation, and leafwise $(1,0)$-form on $M$
\[
\eta := \frac{\pa (\log \sqrt{h_\alpha(z_\alpha, t_\alpha)})}{\pa z_\alpha} dz_\alpha.
\]
Note that they are well-defined on $M$ since (\ref{eq:affine}) implies
\[
h_\alpha = h_\beta \left|\frac{dw_\beta}{dw_\alpha} \right|^2 = h_\beta a_{\alpha\beta}^2
\]
on $U_\alpha \cap U_\beta$, and they enjoy
\[
d\omega = \left(\frac{\pa \sqrt{h_\alpha(z_\alpha, t_\alpha)}}{\pa z_\alpha} dz_\alpha +
\frac{\pa \sqrt{h_\alpha(z_\alpha, t_\alpha)}}{\pa \ol{z}_\alpha} d\ol{z}_\alpha\right) \wedge dt_\alpha
= (\eta + \ol{\eta}) \wedge \omega.
\]
Since $\Theta_h = 0$, Stokes' formula yields
\begin{align*}
0 = \int_M d(\eta  \wedge \omega) &= \int_M  (d\eta\wedge \omega - \eta\wedge d\omega)\\
&=\int_M(\partial\bar\partial \log(\sqrt{h_\alpha})\wedge \omega - \eta\wedge d\omega)\\
&=\int_M \left(-\frac{1}{2} \Theta_h - \eta \wedge ((\eta + \ol{\eta})\wedge \omega)\right)\\
&=\int_M \left(-\frac{1}{2} \Theta_h - \eta \wedge \ol{\eta} \right)\wedge \omega \\
&= -\int_M \eta \wedge \ol{\eta} \wedge \omega.
\end{align*}
This implies that $\eta = 0$ on $M$, hence, $d\omega = 0$ on $M$.

Finally, we shall construct a closed Levi flat hypersurface in $X \setminus M$ by displacing $M$ using the closed 1-form $\omega$. 
Note that, on each foliated chart $U_\alpha$, $h_\alpha$ depends only on $t_\alpha$ since $\omega$ is $d$-closed. 
Let $\rho \colon W \to \R $ be defined by $\rho(z_\alpha, w_\alpha) := \sqrt{h_\alpha(t_\alpha)} u_\alpha$ on $U_\alpha$.
This is a well-defined defining function of $M$ due to (\ref{eq:affine}).
Hence,  $M_\varepsilon := \rho^{-1}(\varepsilon)$ 
are smooth closed real hypersurfaces for $|\varepsilon| \ll 1$.
Moreover, $M_\varepsilon$ is Levi flat since $M_\varepsilon \cap U_\alpha$ is foliated by $\D \times \{ w_\alpha \}$.
This existence of smooth closed Levi flat hypersurface contradicts the 1-convexity of $X \setminus M$ due to the maximum principle of strictly plurisubharmonic function, and completes the proof.
\end{proof}

\begin{Remark}\label{Remark}
When $X$ is algebraic, the Levi foliation of $M$ extends to not only outside of the maximal compact analytic set $A$ but also across $A$ as shown in \cite{Canales},
and the transversely affine structure extends over $X$ as \textit{degenerate} transversely affine structure.
Degenerate transversely affine structure was first studied by Sc\'ardua \cite{Scardua} and defined by Cousin and Pereira in \cite{Cousin-Pereira}. 
A \emph{degenerate} transversely affine structure of a singular holomorphic foliation $\mathcal{F}$ is 
a meromorphic flat connection
$$\nabla\colon \mathcal{O}_X(N_\mathcal{F}) \rightarrow \mathcal{O}_X(N_\mathcal{F}) \otimes \Omega^1_X(*D)\quad \text{such that $\nabla \omega=0$}$$
where $\omega$ is a defining holomorphic $1$-form of $\mathcal{F}$ valued in $N_{\mathcal F}$ and $D$ is a reduced divisor on $X$. $\Omega^1_X(*D)$ denotes the sheaf of meromorphic $1$-forms on $X$ with poles along $D$. See \cite{Cousin-Pereira} for more details. Although this notion was originally called \textit{singular} transversely affine structure in \cite{Cousin-Pereira}, we are following the terminology introduced in \cite{Canales}, that is, \textit{degenerate} transversely affine structure, to avoid confusion with a singular foliation.
\end{Remark}

\section{Examples}
\label{sect:example}
We collect well-known examples of Levi flat hypersurfaces to illustrate our non-existence result.
The first example shows that a compact \emph{non-K\"ahler} surface may contain 
a real analytic closed Levi flat hypersurface 
whose complement is 1-convex and Levi foliation is transversely affine.

\begin{Example}[Hyperbolic Inoue surfaces, cf. \cite{Zaffran}*{p.401}]
Let $A \in SL(2,\Z)$ with $\tr A > 2$. We assume that $A$ is an even Dloussky matrix, that is, for some positive integer $n$, there exist positive integers $k_1, k_2, \dots, k_{2n}$ such that
\[
A = \begin{pmatrix} 0 & 1 \\ 1 & k_1 \end{pmatrix}  \begin{pmatrix} 0 & 1 \\ 1 & k_2 \end{pmatrix} \dots  \begin{pmatrix} 0 & 1 \\ 1 & k_{2n} \end{pmatrix}.
\]
Denote by $e_1, e_2 \in \R^2$ the eigenvectors corresponding to 
the eigenvalues $\lambda > 1$ and $\lambda^{-1}$ of the matrix $A$ as in Example \ref{ex:hyp_torus}.
Regard $e_1, e_2 \in \C^2 = \C \otimes_\R \R^2$, and consider a tube domain $T := \mathbf{H} e_1 \oplus \C e_2 \subset \C^2$, where $\mathbf{H} = \{ z \in \C \mid \Re z > 0 \}$ the right half plane.
Using $\exp \colon \C^2 \to (\C^*)^2$, $\exp(z,w) := (e^{2\pi z}, e^{2\pi w})$, 
we obtain an incomplete Reinhardt domain $\exp(T) \subset \C^2$. 
The matrix $A$ acts linearly on $\C^2$ preserving $T$ since $\lambda > 0$, and induces a $\Z$-action on $\exp(T)$, which is properly discontinuous and fixed point free. 
The quotient $X'' := \exp(T)/\Z$ is an open complex surface, and has a compactification $X'$ with two normal singularities (see, for instance, \cite{Zaffran}*{Proposition 2.2}).
The minimal desingularization $X$ of $X'$ is called a \emph{hyperbolic Inoue surface} (also known as an \emph{even Inoue--Hirzebruch surface}). 
 
 This surface $X$ contains a closed Levi flat hypersurface $M := \exp(\mathbf{H}e_1 \oplus i\R e_2)/\Z$.
The Levi foliation of $M$ is induced from the product foliation of $\mathbf{H}e_1 \oplus i\R e_2$ by right half planes. 
The identification $\R \times \R^2 \to \mathbf{H}e_1 \oplus i\R e_2, (t, xe_1+ye_2) \mapsto  (\lambda^t + ix)e_1 + iy e_2$ induces an isomorphism between the Levi foliation of $M$ and the model foliation of a hyperbolic torus bundle, described in Example \ref{ex:hyp_torus}, hence, the Levi foliation is transversely affine.  

The 1-convexity of the complement of $M$ may follow from \cite{Brunella}*{Proposition 2.1}. 
The defining 1-form $\omega = \lambda^t dy$ of the model foliation induces a hermitian metric of the normal bundle to the Levi foliation, $| \frac{\partial}{\partial y}|^2 = \lambda^{2t} = (\Re z)^2$, where $z = \lambda^t + ix$ is a holomorphic coordinate of leaves. 
Its leafwise curvature form is $\opa_z \pa_z \log (\Re z)^2 = 4dz \wedge \ol{dz}/(\Re z)^2$ and positive everywhere.
\end{Example}

The second and third examples show that compact complex surfaces, regardless of being K\"ahler, may contain a real analytic
closed Levi flat hypersurface whose complement is Stein and Levi foliation is transversely affine \emph{away from some compact leaves}.  
We remark that in these examples the Levi foliation extends to a \emph{degenerate} transversely affine foliation of the ambient space (cf. Remark \ref{Remark}).

\begin{Example}[Affine line bundles compactified by the section at infinity, cf. \cite{Diederich-Ohsawa}]
Let $\Sigma$ be a compact Riemann surface of genus $\geq 1$. 
Take a non-trivial representation $\rho \colon \pi_1(\Sigma) \to \mathrm{Aff^+}(\R) := \{ x \mapsto ax +b \mid a > 0, b \in \R\}$. 
We regard $\mathrm{Aff^+}(\R) \subset \operatorname{Aut}(\mathbb{CP}^1)$.
We obtain a ruled surface $X$ over $\Sigma$ by $X := \Sigma \times_\rho \mathbb{CP}^1 := \widetilde{\Sigma} \times \mathbb{CP}^1 / \thicksim$ 
where $\widetilde{\Sigma}$ is a universal covering of $\Sigma$ and 
we identify $(z, \zeta) \sim (\gamma z, \rho(\gamma)\zeta)$ for $\gamma \in \pi_1(\Sigma)$.
This ruled surface $X$ contains a closed Levi flat hypersurface $M := \Sigma \times_\rho \mathbb{RP}^1$.
Since the holonomy $\rho$ is real affine, the Levi foliation of $M$ is transversely affine away from a compact leaf, $\Sigma \times \{ \infty \}$, the section at infinity. 
A proof of the Steinness of the complement can be found in 
\cite{Barrett}*{Theorem 2, Case 2 of the proof for (g) $\Rightarrow$ (b)}.
\end{Example}

\begin{Example}[Torus bundles, cf. \cite{Nemirovski}]
Let $\Sigma$ be a compact Riemann surface, and $L$ a holomorphic line bundle.
A torus bundle $X$ over $\Sigma$ is obtained by $X : = (L \setminus Z) / \Z$ 
where $Z$ is the zero section of $L$ and $\Z$-action is defined by $\zeta \mapsto 2^n\zeta$ for $\zeta \in L$ and $n \in \Z$.
Assume that there is a meromorphic section $s$ of $L$ having simple zeros and poles. 
Then, $M' := \{ t s(z) \in L \mid t \in \R^*, s(z) \neq 0, \infty \} / \Z$ is a Levi flat hypersurface since it is foliated by the graphs of $ts$ where $t \in \R^*$. 
By taking closure of $M'$ in $X$, which is equivalent to adding elliptic curves fibered over $s^{-1}(\{0, \infty\})$ as compact leaves, 
we obtain a smooth closed Levi flat hypersurface $M$ in $X$.
Note that the Levi foliation of $M$ is transversely affine away from the compact leaves since $M'$ has two connected components, each of which is isomorphic to the product $\Sigma \setminus s^{-1}(\{0,\infty\}) \times S^1$. 
The complement is Stein because each connected component is isomorphic to the product of the open Riemann surface $\Sigma \setminus s^{-1}(\{0,\infty\})$ with an annulus. \end{Example}

\section*{Acknowledgements}
We thank Carolina Canales Gonz\'{a}lez for explaining some details of her work \cite{Canales}. 
We are grateful to Yoshihiko Mitsumatsu and Noboru Ogawa for pointing out inaccuracies in the first draft of this paper,
and to the referees for their thorough remarks.\\

This preprint has not undergone peer review or any post-submission improvements or corrections.\\
The Version of Record of this article is published in Mathematische Zeitschrift, and is available online at https://doi.org/10.1007/s00209-021-02927-z


\end{document}